\documentclass[11pt,reqno]{amsart}

\usepackage{amsmath,amsfonts,amssymb,amsthm,enumerate,url}
\usepackage{microtype}
\usepackage{amssymb,bbm,mathrsfs}

\numberwithin{equation}{section}
\usepackage[colorlinks=true, pdfstartview=FitV, pagebackref=false]{hyperref}
\usepackage{xcolor}
\usepackage{tikz}
\usetikzlibrary{trees,arrows,positioning,automata,shadows,fit,shapes}
\usepackage[numbers,sort]{natbib}
\usepackage[a4paper,vmargin=3cm,inner=2.8cm,outer=2.8cm]{geometry}

\newtheorem{theorem}{Theorem}[section]
\newtheorem{lemma}[theorem]{Lemma}
\newtheorem{proposition}[theorem]{Proposition}

\newtheorem{remark}{Remark}[section]

\newcommand{\R}{\mathbb R}
\newcommand{\N}{\mathbb N}

\DeclareMathOperator{\bP}{\mathbf{P}}
\DeclareMathOperator{\bE}{\mathbf{E}}

\renewcommand{\P}{\mathbb{P}}
\newcommand{\E}{\mathbb{E}}

\newcommand{\bd}{{\mathbf d}}

\newcommand{\cB}{{\mathcal B}}

\newcommand{\cD}{{\mathcal D}}

\newcommand{\cF}{{\mathcal F}}

\newcommand{\cH}{{\mathcal H}}
\newcommand{\cI}{{\mathcal I}}

\newcommand{\cN}{{\mathcal N}}
\newcommand{\cO}{{\mathcal O}}

\newcommand{\cR}{{\mathcal R}}

\newcommand{\cT}{{\mathcal T}}
\newcommand{\cU}{{\mathcal U}}

\renewcommand{\tilde}{\widetilde}
\newcommand{\W}{\ensuremath{W}}

\newcommand{\bw}{{\mathbf w}}
\newcommand{\bh}{{\mathbf h}}
\newcommand{\bX}{{\mathbf X}}
\newcommand{\bphi}{{\mathbf \Phi}}

\DeclareMathOperator{\var}{Var}
\DeclareMathOperator{\cov}{Cov}

\newcommand{\given}{\;\big|\;}
\newcommand{\ind}{{\mathbbm{1}}}
\newcommand{\tmix}{t_\textsc{mix}}
\newcommand{\tv}{{\textsc{tv}}}
\newcommand{\nbrw}{{\textsc{nbrw}}}

\newcommand{\red}{{\textsc{red}}}

\usepackage{color}

\begin{document}

\title[A threshold for cutoff]{A threshold for cutoff in two-community\\ random graphs}

\author[ ]{Anna Ben-Hamou}
\address{ \hfill\break
A. Ben-Hamou\hfill\break
Sorbonne Universit\'e, LPSM\hfill\break
4, place Jussieu, 75005 Paris, France.}
\email{anna.ben-hamou@upmc.fr}

\begin{abstract}
In this paper, we are interested in the impact of communities on the mixing behavior of the non-backtracking random walk. We consider sequences of sparse random graphs of size $N$ generated according to a variant of the classical configuration model which incorporates a two-community structure. The strength of the bottleneck is measured by a parameter $\alpha$ which roughly corresponds to the fraction of edges that go from one community to the other. We show that if $\alpha\gg \frac{1}{\log N}$, then the non-backtracking random walk exhibits cutoff at the same time as in the one-community case, but with a larger cutoff window, and that the distance profile inside this window converges to the Gaussian tail function. On the other hand, if $\alpha \ll \frac{1}{\log N}$ or $\alpha \asymp \frac{1}{\log N}$, then the mixing time is of order $1/\alpha$ and there is no cutoff.
\end{abstract}


\maketitle
\section{Introduction}

\subsection{Setting.}

We consider an extension of the classical configuration model, designed to incorporate a two-community structure. Let $V$ be a vertex set partitioned into two non-empty \emph{communities} $V_0$ and $V_1$, i.e.
$$
V=V_0\cup V_1 \qquad \textrm{ and } \qquad V_0\cap V_1=\emptyset\, .
$$
Let $\bd :V\to \N\setminus\{0,1\}$ be a fixed degree sequence such that
$$
\sum_{v\in V_0}\bd(v)=N_0, \quad \textrm{ and } \quad \sum_{v\in V_1}\bd(v)=N_1
$$
are both even. Let $N=N_0+N_1$. Each vertex $v$ of $V$ is endowed with $\bd(v)$ \emph{half-edges}, and, for $i=0,1$, we denote by $\cH_i$ the set of half-edges attached to a vertex of $V_i$, and let $\cH=\cH_0\cup\cH_1$. By definition, $|\cH_0|=N_0$, $|\cH_1|=N_1$ and $|\cH|=N$. 

Now let $p$ be a fixed even integer between $2$ and $\min \{N_0, N_1\}$ and choose uniformly at random $p$ distinct half-edges in $\cH_0$ to form the random subset of \emph{outgoing} half-edges of $\cH_0$. Similarly, and independently, choose uniformly at random $p$ distinct half-edges in $\cH_1$ to form the random subset of outgoing half-edges of $\cH_1$. Half-edges which are not outgoing are called \emph{internal} half-edges. Let $\alpha_0=p/N_0$, $\alpha_1=p/N_1$ and $\alpha=\alpha_0+\alpha_1$.

We now generate the random graph $G$ by choosing independently a uniform pairing on the set of internal half-edges of $\cH_0$, a uniform pairing on the set of internal half-edges of $\cH_1$ (this is feasible since both sets have even size), and a uniform matching between the set of outgoing half-edges of $\cH_0$ and the set of outgoing half-edges of $\cH_1$ (which have equal size $p$). We let $\eta$ be the induced pairing on $\cH$. If $x$ and $y$ are two distinct half-edges attached to vertices $u$ and $v$ respectively, then the pairing $\eta(x)=y$ (which is equivalent to $\eta(y)=x$) induces an edge between $u$ and $v$ in the resulting graph.

\begin{remark}
The graph model may very well be defined with $N_0$, $N_1$ and $p$ all odd, the important thing being that $N_0-p$ and $N_1-p$ are both even. However, assuming that $p$ is even is quite convenient for the analysis, in particular in Section~\ref{sec:no-cutoff}.
\end{remark}

We are interested in the mixing properties of the non-backtracking random walk (\nbrw) on $G$, defined as the Markov chain with state space $\cH$ and transition matrix
$$P(x,y)=
\begin{cases}
\frac{1}{\deg(\eta(x))} & \textrm{if $y$ and $\eta(x)$ are neighbors,} \\
0 & \textrm{ otherwise},
\end{cases}
$$
where two half-edges $x$ and $y$ are called \emph{neighbors} if they are attached to the same vertex and are different. The \emph{degree} of half-edge $x$, denoted $\deg(x)$, corresponds to the number of neighbors of $x$ (if $x$ is attached to vertex $u$, then $\deg(x)=\bd(u)-1$). The \nbrw\ thus moves at each step from the current state $x$ to a uniformly chosen neighbor of $\eta(x)$. (Equivalently, the \nbrw\ can be defined over the directed edges of the graph: at each step, it moves from 
$(u,v)$ to $(v,u')$, with $u\neq u'$, with probability $\frac{1}{\bd(v)-1}$. In particular, the \nbrw\ cannot traverse the same edge twice in a row in opposite directions.)

The matrix $P$ enjoys the following symmetry property with respect to $\eta$: for all $x,y\in\cH$, 
  \begin{equation}
\label{sym}
P(\eta(y),\eta(x)) = P(x,y)\, .
\end{equation}
In particular, $P$ is doubly stochastic and the stationary distribution of the chain is the uniform distribution $\pi$ on $\cH$. The worst-case total-variation distance to equilibrium at time $t\geq 0$ is 
 \begin{eqnarray*}
\cD(t)   =  \max_{x\in\cH}\cD_x(t),\quad \textrm{where}\quad\cD_x(t)=\sum_{y\in\cH}\left( \frac{1}{N}- P^t(x,y)\right)_+\, .
\end{eqnarray*}
This quantity is weakly decreasing in $t$, and the first time when it falls below a given threshold $0<\varepsilon<1$ is the $\varepsilon$-mixing time:
$$
\tmix(\varepsilon) = \inf\left\{t\geq 0,\,  \cD(t)<\varepsilon\right\}\, .
$$

\subsection{Results.}
Let
\begin{equation}
\label{eq:def-mu-sigma}
\mu=\frac{1}{N}\sum_{x\in\cH} \log\deg(x) \quad\textrm{and}\quad\sigma^2=\frac{1}{N}\sum_{x\in\cH} (\log\deg(x)-\mu)^2
\end{equation}
be the mean and variance of the logarithmic degree of a uniformly chosen half-edge. For $i\in\{0,1\}$, let also
\begin{equation}
\label{eq:def-mu0-mu1}
\mu_i=\frac{1}{N_i}\sum_{x\in\cH_i} \log\deg(x) \quad\textrm{and}\quad\sigma_i^2=\frac{1}{N_i}\sum_{x\in\cH_i} (\log\deg(x)-\mu_i)^2
\end{equation}
be the mean and variance of the logarithmic degree within community $\cH_i$.

We consider a sequence $(G_n)_{n\geq 1}$ of graphs distributed according to this model, with $N\to\infty$ as $n\to\infty$ (the index $n$ will be omitted from notation) and are interested in the following regime:
\begin{subequations}
\label{assumptions}
    \begin{align}
        & \alpha_1+\alpha_0\leq 1\,  &\mbox{{\small (there is a community structure)}}   \label{eq:assumption:p} \\
      & N_0\asymp N_1 \asymp N\,  &\mbox{{\small (communities have comparable size)}}   \label{eq:assumption-N0-N1}\\
      & \liminf \sigma^2>0\, , &\mbox{{\small (non-vanishing variance)}} \label{eq:assumption-sigma}\\
      & \min_{v\in V} \bd(v)\geq 3\,  &\mbox{{\small (branching degrees)}}   \label{eq:assumption-min-degree}\\
      & \Delta=\max_{v\in V} \bd(v)=O(1)\,  &\mbox{{\small (sparse regime)}}   \label{eq:assumption-Delta}
    \end{align}
\end{subequations}
To see why condition~\eqref{eq:assumption:p} corresponds to the presence of community structure, observe that when $\eta$ is a uniform pairing over $\cH$, then the expected number of pairs between one element of $\cH_0$ and one element of $\cH_1$ is equal to $\frac{N_0N_1}{N-1}$. In expectation, the analogue of $\alpha_0+\alpha_1$ is then equal to $\frac{N}{N-1}\approx 1$.

For the first part of our result, we need the following additional assumption
\begin{equation}\label{eq:assumption:mu_0-mu_1}
\text{Either } \quad \liminf|\mu_0-\mu_1|>0 \quad \text{ or } \quad |\mu_0-\mu_1|^2=o(\alpha) \, .
\end{equation}

\begin{theorem}\label{thm:cutoff}
Under assumptions~\eqref{assumptions} and~\eqref{eq:assumption:mu_0-mu_1}, if $\alpha\gg \frac{1}{\log N}$, then for all $\varepsilon\in (0,1)$,
$$
\frac{\tmix(\varepsilon)-\frac{\log N}{\mu}}{\sqrt{\frac{\nu^2\log N}{\mu^3}}} \overset{\P}\longrightarrow  \overline{\Phi}^{-1}(\varepsilon)\, ,
$$
where
$$
\nu^2=\sigma^2+\frac{2N_0N_1(1-\alpha)}{N^2}\frac{(\mu_0-\mu_1)^2}{\alpha}\, ,
$$
and $\overline{\Phi}$ is the tail function of the standard normal distribution ($\overline{\Phi}=1-\Phi$ with $\Phi$ the c.d.f. of the standard normal distribution).
\end{theorem}

\begin{theorem}\label{thm:no-cutoff}
Under assumptions~\eqref{assumptions}, if $\alpha\ll \frac{1}{\log N}$ or $\alpha\asymp \frac{1}{\log N}$, then for all $x\in\cH_i$ and for all $\varepsilon<\frac{N_{1-i}}{N}$, there exist $a, b>0$ depending only on $\varepsilon$ and $N_i/N$ such that
$$
\frac{a+o_\P(1)}{\alpha}\leq \tmix^{(x)}(\varepsilon)\leq \frac{b+o_\P(1)}{\alpha}\, \cdot 
$$
and there is no cutoff.
\end{theorem}

Let us briefly comment on the results. It is natural to expect that the presence of communities has an influence on the mixing behavior of the \nbrw. If $\alpha$ is very small, i.e. if there are only few edges that go from one community to the other, then the graph has a very narrow bottleneck and the walk will take a long time to cross this bottleneck. Intuitively, the mixing time in this case is determined by the geometric time needed to hit one of those crossing edges, and the distance then decreases smoothly, as the tail function of a Geometric variable: there is no cutoff. 

On the other hand, if $\alpha$ is large, then the walk can easily go from one community to the other, and the mixing behavior is very similar to the case where there is no community structure, as studied  by B. and Salez~\citep{ben-salez}. In this paper, the authors considered the configuration model with $\eta$ uniformly chosen among all possible pairings on $\cH$. They showed, under much weaker degree assumptions, that the \nbrw\ has cutoff at time $\frac{\log N}{\mu}$, with window $\sqrt{\frac{\sigma^2}{\mu^3}\log N}$ and that the distance profile inside the window is Gaussian. 

The contribution of the present paper is to determine quite precisely the threshold between those two regimes, the one with no community structure and the one with two communities connected by very few edges. As it turns out, cutoff can still occur with a strong community structure, even in a regime where the proportion $\alpha$ of crossing edges vanishes to $0$, provided it decays more slowly than $1/\log N$. This threshold arises as the result of a competition between the mixing time in each community, which is of order $\log N$, and the time it takes to switch community, which is approximately Geometric with expectation of order $1/\alpha$. This result can be interpreted in light of a series of powerful results that relate mixing and hitting times~\citep{oliveira2012mixing,sousi2013hitting,griffiths2014tight} and that characterize cutoff in terms of concentration of hitting times of ``worst'' sets~\citep{basu2015characterization,hermon2015technical}. 

Another interesting fact is the impact of communities on the cutoff window (in the regime $\alpha \gg 1/\log N$). In the case of no community structure, the window is of order $\sqrt{\frac{\sigma^2}{\mu^3}\log N}$, which, under assumptions~\eqref{assumptions}, has order $\sqrt{\log N}$. The introduction of a community structure can significantly increase the cutoff window. Under our assumptions, this window is of order $\sqrt{\frac{\log N}{\alpha}}$, which is still much smaller than $\log N$, the first order of the mixing time, but can be much larger than $\sqrt{\log N}$. Let us also note that, for some fixed value of $\alpha$, the window is maximized for $\alpha_0=\alpha_1$, i.e. for $N_0=N_1$, when the two communities have equal size.

\subsection{Related work.}
A sequence of chains $(P_n)$ is said to exhibit the cutoff phenomenon if for all $\varepsilon\in (0,1)$, $\tmix^{(n)}(\varepsilon)\sim\tmix^{(n)}(1-\varepsilon)$ as $n\to\infty$. In other words, convergence to equilibrium occurs very abruptly: the total-variation distance drops from near $1$ to near $0$ at the mixing time, over a much shorter time known as the cutoff window. It was first observed for random walks on finite groups, such as random transpositions on the symmetric group~\citep{diaconis1981generating}, or the lazy random walk on the hypercube~\citep{aldous1983mixing}. It was then observed in various other contexts, such as the Glauber dynamics on the Ising model at high temperature~\citep{lubetzky2017universality}, or the simple exclusion process~\citep{labbe2016cutoff}. This phenomenon was quickly conjectured to be a widespread phenomenon, satisfied by a large class of finite Markov chains. However, finding simple sufficient conditions for cutoff appeared to be a very challenging task and several conditions that appeared to be ``natural'' have been disproved by counter-examples. For instance, regular expanders of bounded degree have remarkable mixing properties and one could reasonably expect that the (lazy) random walk on such graphs has cutoff, but this was disproved in~\citep{lubetzky2011explicit}. However, one can rather ask: what is the mixing behavior of the random walk on a ``typical'' graph? This led to study random walks on random graphs, uniformly chosen in a given class. In this line of work, the article of \citet{lubetzky2010cutoff} was a breakthrough: they showed that, with high probability, simple and non-backtracking random walks on random $d$-regular graphs have cutoff. Cutoff for \nbrw\ was then established on sparse random graphs with given degrees, by B. and Salez~\citep{ben-salez}, and independently by~\citet{berestycki2018random}, and \citet{bordenave2018random} established the cutoff phenomenon for the random walk on sparse random directed graphs. Recently, \cite{avena2018mixing} \citep{avena2018random} studied \nbrw\ on dynamical random graphs, and established three different mixing behavior according to the rate at which the graph is re-randomized. 

Those random graph models are ``homogeneous'' in the sense that with high probability, they do not give rise to a community structure within vertices. However, various real networks, such as social or biological networks~\citep{girvan2002community}, exhibit a community structure: there is a partition of vertices such that vertices in the same group are more likely to be connected than vertices in different groups. Probably one of the most famous random graph model with a community structure is the \emph{stochastic block model}. This model was first introduced by~\citep{holland1983stochastic}, and was then studied in a wide variety of contexts, in particular in the very rich research area of \emph{community detection} (see~\citep{abbe2018community} for a survey of recent results). Fixed-degree variants of the stochastic block model, often referred to as \emph{hierarchical configuration models}, were introduced and investigated by~\citep{sah2014exploring}, \citep{van2015hierarchical} and~\citep{stegehuis2016epidemic}, with a particular focus on epidemic propagation. The model considered here can be seen as a variant of the hierarchical configuration model with two communities, where randomization is used first to determine which half-edges are outgoing, and then to choose the pairings of internal and outgoing half-edges (degrees, however, are fixed). In his master thesis,  Poir\'ee~\citep{thomas} studied \nbrw\ on such random graphs, in the particular case of regular degrees and communities of equal size.

\subsection{Open questions.}
Several extensions of the model would be interesting to investigate and a lot of related questions could be raised. Let us briefly mention some of them:
\begin{itemize}
\item The regime in \eqref{assumptions} is quite restrictive, it would be interesting to see how far those assumptions can be relaxed (in the one-community case, B. and Salez~\citep{ben-salez} could go up to $\Delta=N^{o(1)}$). We expect that Theorem~\ref{thm:cutoff} holds without assumption~\eqref{eq:assumption:mu_0-mu_1} nor any other assumption on $|\mu_0-\mu_1|$. Nevertheless, condition~\eqref{eq:assumption:mu_0-mu_1} encompasses a variety of situations, including the important case where degrees are \textsc{i.i.d.} according to some distribution over $\{3,\dots, \Delta\}$, in which case the Central Limit Theorem yields $|\mu_0-\mu_1|=O_\P(n^{-1/2})$.
\item Instead of choosing the outgoing half-edges at random, it would be interesting to consider the model where each vertex initially has a fixed number of outgoing and internal half-edges.
\item What happens with more than two communities? Consider for instance the following simple generalization of the model with $K\geq 2$ communities of equal size $N/K$: in each community $\cH_k$, $k\in\{1,\dots, K\}$, choose uniformly at random $K-1$ distinct blocs $(\cB_{k,\ell})_{\ell\neq k}$ of $p$ distinct half-edges. Then, for all $k$ and $\ell\neq k$, choose a uniform matching between $\cB_{k,\ell}$ and $\cB_{\ell,k}$, and a uniform pairing over $\cH_k\setminus (\cup_{\ell\neq k}\cB_{k,\ell})$. Letting $\alpha=pK^2/N$, we expect (at least in when $K$ is fixed) that if $\alpha\gg \frac{1}{\log N}$, then there is cutoff at $\frac{\log N}{\mu}$ with window $\sqrt{\frac{\nu^2}{\mu^3}\log N}$ where
\[
\nu^2= \sigma^2 +\frac{2(1-\alpha)}{K^2 \alpha} \left\{ (K-1) \sum_{k=1}^K \mu_i^2 -\sum_{k\neq \ell} \mu_k\mu_\ell \right\}\, .
\]
\item What happens for the simple random walk?
\end{itemize}

\section{A useful coupling.}
\label{sec:coupling}

Before entering into the proofs, we describe a coupling for typical non-backtracking trajectories, which helps approximating the annealed law of the walk and will be crucially used later on. This coupling takes advantage of the fact that the \nbrw\ started at a given $x\in\cH$ and the graph along its trajectory can be generated simultaneously as follows: initially, $X_0=x\in\cH$, all half-edges are unpaired and no type has been allocated yet (the property of a half-edge to be outgoing or internal will be referred to as its type); then at each time $k\geq 0$, 
\begin{enumerate}
\item 
\begin{enumerate}
\item if the type of $X_{k}$ has not been fixed yet and if $X_{k}$ belongs to $\cH_i$ for $i=0,1$, we make $X_{k}$ outgoing with probability $\alpha_{i}^{(k)}$ corresponding to the conditional probability that $X_k$ is outgoing given the past. With probability $1-\alpha_{i}^{(k)}$, we make $X_{k}$ internal;
\begin{enumerate}
\item if $X_{k}$ is outgoing, we pair it with a uniformly chosen unpaired half-edge of $\cH_{1-i}$ and declare that this chosen half-edge is outgoing;
\item if $X_{k}$ is internal, we pair it with a uniformly chosen other unpaired half-edge of $\cH_{i}$ and declare that this chosen half-edge is internal;
\end{enumerate}
\item if the type of $X_{k}$ has already been set, then $\eta(X_{k})$ is already defined and no new pair is formed;
\end{enumerate} 
\item in both cases, once $\eta(X_{k})$ is determined, its neighbors are deterministically given by the degree sequence, and we let $X_{k+1}$ be a uniformly chosen neighbor of $\eta(X_{k})$.
\end{enumerate}
 The sequence $\{X_k\}_{k\geq 0}$ is then exactly distributed according to the annealed law. Now, consider a sequence $\{X_k^\star\}_{k\geq 0}$ generated in the following way: initially $X_0^\star=x\in\cH$; then at each time $k\geq 0$, 
\begin{enumerate}
\item if $X_{k}^\star$ belongs to $\cH_i$ for $i=0,1$, draw a Bernoulli random variable $B_k$ with parameter $\alpha_i=p/N_i$;
\begin{enumerate}
\item if $B_k=1$, let $\eta(X_k^\star)$ be a uniformly chosen half-edge in $\cH_{1-i}$;
\item if $B_k=0$, let $\eta(X_k^\star)$ be a uniformly chosen half-edge in $\cH_{i}$;
\end{enumerate} 
\item in both cases, let $X_{k+1}^\star$ be a uniformly chosen neighbor of $\eta(X_{k}^\star)$.
\end{enumerate}
 
The process $\{X_k\}_{k\geq 1}$ and the simple Markov chain $\{X_k^\star\}_{k\geq 1}$ can be coupled in such a way that the two sequences are equal up to the first time $k$ where either the types of $X_k$ and $X_k^\star$ differ, or the two types are equal but the uniformly chosen half-edge $\eta(X_{k}^\star)$ is already paired. The total-variation distance between the type indicators at time $k$ is smaller than $\max_{i=0,1}|\alpha_i^{(k)}-\alpha_i |$. Using the facts that at least $(p-k)\vee 0$ half-edges remain to be made outgoing in each community, that there are at least $N_i-2k$ unpaired half-edges in $\cH_i$, and that $p\leq \min\{N_0, N_1\}$, we have, for all $k <\min\{N_0, N_1\}/2$,
$$
\frac{-k}{N_i}\leq \alpha_i^{(k)}-\alpha_i \leq \frac{2k}{N_i-2k}\, \cdot 
$$
Also, as there are less than $2k$ paired half-edges by step $k$, the probability that $\eta(X_{k}^\star)$ is already paired is less than $\max_{i=0,1} 2k/N_i$. Letting $T$ be the first time where the two coupled sequence differ and using a crude union-bound yields
\begin{equation}
\label{eq:coupling}
\P\left(T\leq t\right)=O\left(\frac{t^2}{N}\right)\, ,
\end{equation}
by~\eqref{eq:assumption-N0-N1}.  The distribution of $\{X_k^\star\}_{k\geq 1}$ is much simpler than that of $\{X_k\}_{k\geq 1}$: at each step, draw a Bernoulli random variable whose parameter depends on the current community. If it is equal to $1$, move to a uniform half-edge from the other community; if it is equal to $0$, move to a uniform half-edge from the same community. It is not hard to check that the stationary distribution of this Markov chain is uniform over $\cH$.

Letting $S_t=\sum_{k=1}^t \log\deg(X_k^\star)$, we have the following Central Limit Theorem: for all $x\in\cH$ and $\lambda\in\R$,
\begin{eqnarray*}
\P_x\left(\frac{S_t -\mu t}{\nu\sqrt{t}}\leq \lambda\right) & \underset{t\to\infty}\longrightarrow & \Phi(\lambda)\, ,
\end{eqnarray*}
where 
\begin{eqnarray*}
\nu^2&=& \lim_{t\to\infty}\frac{1}{t} \var_{\pi}(S_t)= \sigma^2 +2\sum_{s=1}^{+\infty} \cov_{\pi}\left(\log\deg(X_0^\star),\log\deg(X_s^\star)\right)\, .
\end{eqnarray*}
In the definition above, the subscript $\pi$ means that $X_0^\star\sim\pi$. We have
\begin{eqnarray*}
\nu^2&=& \sigma^2 +2 \sum_{x,y\in\cH}\frac{1}{N}\sum_{s=1}^{+\infty} \left(\P_x(X_s^\star=y)-\frac{1}{N}\right)\log\deg(x)\log\deg(y)\, .
\end{eqnarray*}
Note that for all $i,j\in\{0,1\}$, for all $x\in\cH_i$ and $y\in\cH_j$, we have 
$$
\P_x(X_s^\star =y)=\frac{\P_{\pi_i}(X_s^\star\in \cH_j)}{N_j}\, ,
$$
where $\pi_i$ is the uniform distribution over $\cH_i$, hence
\[
\nu^2=\sigma^2 +2\sum_{(i,j)\in\{0,1\}^2}\frac{N_i\mu_i\mu_j}{N}\sum_{s=1}^{+\infty} \left( \P_{\pi_i}(X_s^\star\in \cH_j)-\frac{N_j}{N}\right)\, .
\]
Noticing that the sequences $\left(\P_{\pi_0}(X_s^\star \in\cH_0)\right)_{s\geq 0}$ and $\left(\P_{\pi_1}(X_s^\star \in\cH_1)\right)_{s\geq 0}$ obey the following induction relations
$$
\left\{
\begin{array}{ll}
\P_{\pi_0}(X_s^\star \in\cH_0)&\!\!\!= (1-\alpha_0)\P_{\pi_0}(X_{s-1}^\star \in\cH_0)+\alpha_0\left(1-\P_{\pi_1}(X_{s-1}^\star \in\cH_1)\right)\, ,\\
\P_{\pi_1}(X_s^\star \in\cH_1)&\!\!\!= (1-\alpha_1)\P_{\pi_1}(X_{s-1}^\star \in\cH_1)+\alpha_1\left(1-\P_{\pi_0}(X_{s-1}^\star \in\cH_0)\right)\, ,
\end{array}
\right.
$$
we obtain
\begin{equation}
\label{eq:prob-coupling}
\left\{
\begin{array}{ll}
\P_{\pi_0}(X_s^\star \in\cH_0)&= \frac{\alpha_0(1-\alpha_0-\alpha_1)^s+\alpha_1}{\alpha_0+\alpha_1}=\frac{N_0}{N} + \frac{N_1}{N}(1-\alpha)^s\, ,\\
\P_{\pi_1}(X_s^\star \in\cH_1)&= \frac{\alpha_1(1-\alpha_0-\alpha_1)^s+\alpha_0}{\alpha_0+\alpha_1}=\frac{N_1}{N} + \frac{N_0}{N}(1-\alpha)^s\, ,
\end{array}
\right.
\end{equation}
which yields
\begin{equation}\label{eq:def-v}
\nu^2=\sigma^2+\frac{2N_0N_1(1-\alpha)}{N^2}\frac{(\mu_0-\mu_1)^2}{\alpha}\, \cdot 
\end{equation}
We will also need a quantitative control on the CLT normal approximation, in the form of Berry-Esseen type bound.

\begin{lemma}\label{lem:berry-esseen}
Under assumptions~\eqref{assumptions} and~\eqref{eq:assumption:mu_0-mu_1}, for all $x\in\cH$ and $t \gg \frac{1}{\alpha}$,
\begin{eqnarray*}
\sup_{\lambda\in\R}\left| \P_x\left(\frac{S_t -t\mu}{\nu\sqrt{t}}\leq \lambda\right) -\Phi(\lambda) \right| &=& o(1)\, \cdot 
\end{eqnarray*}
\end{lemma}
\begin{proof}[Proof of Lemma~\ref{lem:berry-esseen}]
Assume first that $\limsup |\mu_0-\mu_1|>0$.
By~\citet[Part I, Chapter 3, Theorem 3.1]{lezaud:tel-01084797} (see also \citet{mann}), we have
\begin{eqnarray*}
\sup_{\lambda\in\R}\left| \P_x\left(\frac{S_t -\mu t}{\nu\sqrt{t}}\leq \lambda\right) -\Phi(\lambda) \right| &\leq & \frac{159\log(\Delta)\sigma^2a_x}{\nu^3 \gamma_\star^2\sqrt{t}}\, ,
\end{eqnarray*}
where $a_x$ is the $\ell^2(\pi)$-norm of the ratio between the distribution of $X_1^\star$ starting from $x$ and $\pi$, i.e.
$$
a_x =\sqrt{N\sum_{y\in\cH}\P_x(X_1^\star=y)^2},
$$
and $\gamma_\star$ is the spectral gap of the chain $(X_k^\star)$.  If $x\in\cH_i$ for $i=0,1$, then
$$
a_x =\sqrt{N\left(\frac{(1-\alpha_i)^2}{N_i}+\frac{\alpha_i^2}{N_{1-i}}\right)}\leq \sqrt{N\left(\frac{1}{N_0}+\frac{1}{N_1}\right)}=O(1)\, ,
$$
since $N_0\asymp N_1 \asymp N$ by assumption~\eqref{eq:assumption-N0-N1}. The second largest eigenvalue of the transition matrix of $(X_k^\star)$ is equal to $1-\alpha$, i.e. $\gamma_\star=\alpha$. Also, by assumption~\eqref{eq:assumption-Delta}, $\Delta=O(1)$ and $\sigma^2=O(1)$.Using that $N_0\asymp N_1\asymp N$, we obtain
$$
\frac{1}{\nu^3 \gamma_\star^2\sqrt{t}} = \frac{(\alpha t)^{-1/2}}{(\alpha \nu^2)^{3/2}}\lesssim \frac{(\alpha t)^{-1/2}}{\left(\alpha\sigma^2+(1-\alpha)(\mu_0-\mu_1)^2\right)^{3/2}}\, ,
$$
and the proof is concluded by assumption~\eqref{eq:assumption-sigma} and the fact that $\limsup |\mu_0-\mu_1|>0$.

Let us now consider  the case $|\mu_0-\mu_1|^2=o(\alpha)$. Note that this implies $\sigma^2\sim \nu^2$. To prove Lemma~\ref{lem:berry-esseen} in this case, we use Stein's method of exchangeable pairs. To alleviate notation, let $Y_s=\log\deg(X_s^\star)$. 
Let
\[
W=\sum_{s=1}^t \frac{Y_s-\mu}{\nu\sqrt{t}}=\frac{S_t-t\mu}{\nu\sqrt{t}}\, ,
\]
and let $W'$ be constructed as follows: choose an index $S\in\{1,\dots,t\}$ uniformly at random. If $X_S^\star\in\cH_i$, for $i=0,1$, then let $\tilde{Y}_S=\log\deg(\tilde{X}_S^\star)$ with $\tilde{X}_S^\star$ chosen uniformly at random in $\cH_i$ and let
\[
W'= \frac{\tilde{Y}_S^\star-\mu}{\nu\sqrt{t}} +\sum_{\substack{s=1\\ s\neq S }}^t \frac{Y_s-\mu}{\nu\sqrt{t}}\, \cdot 
\]
The pair $(W,W')$ is exchangeable and
\begin{align}\label{eq:stein-pair}
\E\left[ W'-W\given X^\star_1,\dots,X^\star_t\right]&= \frac{1}{\nu\sqrt{t}} \sum_{s=1}^t\frac{1}{t}\left\{ \ind_{X_s^\star\in\cH_0}(\mu_0-Y_s)+ \ind_{X_s^\star\in\cH_1}(\mu_1-Y_s)\right\}\nonumber\\
&= -\frac{W}{t}+ \frac{1}{\nu\sqrt{t}} \sum_{s=1}^t\frac{1}{t}\left\{ \ind_{X_s^\star\in\cH_0}(\mu_0-\mu)+ \ind_{X_s^\star\in\cH_1}(\mu_1-\mu)\right\}\nonumber\\
&= \frac{-W+\xi_t}{t}\, ,
\end{align}
where 
\[
\xi_t=\frac{\mu_0-\mu_1}{\nu\sqrt{t}} \left(\sum_{s=1}^t\ind_{X_s^\star \in\cH_0} -\frac{tN_0}{N}\right)\, .
\]
By~\citet[Theorem 3.1]{ross2011fundamentals}, 
\[
d\left(W, \cN(0,1)\right)\leq \sup_{f\in\cF} \Big| \E\left[ f'(W)-Wf(W)\right] \Big|\, ,
\]
where $d(\cdot,\cdot)$ is the Wasserstein metric, and where
\[
\cF=\left\{ f\, :\, \|f\|_{\infty},\|f''\|_{\infty}\leq 2\, ,\, \|f'\|_{\infty}\leq \sqrt{2/\pi}\right\}\, .
\]
Let $f\in\cF$ and let $F$ be a primitive of $f$. By exchangeability and Taylor expansion, there is $W^\star$ between $W$ and $W'$ such that
\begin{align*}
0&=\E\left[F(W')-F(W)\right]\\
&=\E\left[ (W'-W)f(W)+ \frac{1}{2}(W'-W)^2f'(W)+\frac{1}{6}(W'-W)^3f''(W^\star)\right]\, .
\end{align*}
Using~\eqref{eq:stein-pair}, rearranging, and using the assumptions on $f$, we get
\begin{align*}
 \Big| \E\left[ f'(W)-Wf(W)\right] \Big| &=\left| \E\left[\xi_t f(W) +\left(\frac{t}{2}(W'-W)^2 -1\right)f'(W) +\frac{t}{6}(W'-W)^3 f''(W^\star) \right] \right|\\
 &\leq 2\E\left[|\xi_t|\right] +\sqrt{2/\pi} \E\left[\left| \frac{t}{2}\E[(W'-W)^2\given \bX^\star] -1\right|\right] +\frac{2t}{6}\E\left[ |W'-W|^3\right]\, ,
\end{align*}
where $\bX^\star=(X_1^\star,\dots,X_t^\star)$. Let us bound each of the three terms in the sum above separately. First,
\[
t\E\left[ |W'-W|^3\right]\leq \frac{t(\log\Delta)^3}{\nu^3 t^{3/2}}=O\left(\frac{1}{\sqrt{t}}\right)\, .
\]
Moreover, $\E\left[|\xi_t|\right]\leq \E\left[\xi_t^2\right]^{1/2}$. Assume without loss of generality that $X_0^\star=x\in\cH_0$. Then by~\eqref{eq:prob-coupling},
\begin{equation}\label{eq:stein-proportions}
\E_x\left[\sum_{s=1}^t \ind_{X_s^\star\in\cH_0}\right]=\frac{tN_0}{N}+\frac{N_1(1-\alpha)(1-(1-\alpha)^t)}{N\alpha}=\frac{tN_0}{N}+O\left(\frac{1}{\alpha}\right)\, ,
\end{equation}
and
\begin{align*}
\E_x\left[\left(\sum_{s=1}^t \ind_{X_s^\star\in\cH_0}\right)^2\right]&=\sum_{s=1}^t \P_x(X_s^\star\in\cH_0)+2\sum_{s<s'}\left(\frac{N_0}{N}+\frac{N_1}{N}(1-\alpha)^s\right)\left(\frac{N_0}{N}+\frac{N_1}{N}(1-\alpha)^{s'-s}\right)\\
&= \left(\frac{tN_0}{N}\right)^2 +O\left(\frac{t}{\alpha}\right)\, .
\end{align*}
This entails
\[
\E\left[\xi_t^2\right]^{1/2}=O\left( \frac{|\mu_0-\mu_1|}{\sqrt{\alpha}}\right) =o(1)\, .
\]
Finally, 
\[
\E\left[\left| \frac{t}{2}\E[(W'-W)^2\given \bX^\star] -1\right|\right] \leq \left( \frac{t^2}{4} \E\left[\E[(W'-W)^2\given \bX^\star]^2\right] -t \E[(W'-W)^2] +1\right)^{1/2}\, .
\]
Using~\eqref{eq:stein-proportions} and the fact that $\frac{N_0}{N}\sigma_0^2+\frac{N_1}{N}\sigma_1^2\sim \sigma^2\sim\nu^2$, we obtain
\begin{align*}
t \E[(W'-W)^2]&=\frac{2}{\nu^2t}\sum_{s=1}^t \left(\sigma_0^2\P_x(X_s^\star \in\cH_0)+ \sigma_1^2\P_x(X_s^\star \in\cH_1)\right)\\
& = 2+o(1)+ O\left(\frac{1}{t\alpha}\right)\, ,
\end{align*}
and, using that
\[
\sum_{1\leq s<s'\leq t} \P_x(X_s^\star\in\cH_i, X_{s'}^\star\in\cH_j)=\frac{t^2}{2}\frac{N_iN_j}{N^2}+O\left(\frac{t}{\alpha}\right)\, ,
\]
we get
\begin{align*}
 \frac{t^2}{4} \E\left[\E[(W'-W)^2\given \bX^\star]^2\right]&=\frac{1}{4\nu^4t^2} \E\left[\left(\sum_{s=1}^t \sum_{i=0,1}\frac{\ind_{X_s^\star\in\cH_i} }{N_i}\sum_{y\in\cH_i}\left(\log\deg(y)-Y_s\right)^2\right)^2\right]\\
 &= O\left(\frac{1}{t}\right)+\frac{2}{\nu^4 t^2}\sum_{s<s'}\sum_{i,j=0,1} \P_x(X_s^\star\in\cH_i, X_{s'}^\star\in\cH_j)\sigma_i^2\sigma_j^2\\
 &= \frac{1}{\nu^4}\left(\frac{N_0}{N}\sigma_0^2+\frac{N_1}{N}\sigma_1^2\right)^2 +O\left(\frac{1}{\alpha t}\right)\\
 &=1+o(1)+O\left(\frac{1}{t\alpha}\right)\, .
\end{align*}
Hence, if $t\alpha\to \infty$, the Wasserstein distance between $W$ and $\cN(0,1)$ tends to $0$, and so does the Kolmogorov distance.
\end{proof}

\section{Proof of Theorem~\ref{thm:cutoff}}
\label{sec:cutoff}

\subsection{Lower bound.}

Let $x\in \cH$ be a fixed starting point and let
$$
t=\frac{\log N}{\mu}+(\lambda+o(1))\sqrt{\frac{\nu^2}{\mu^3}\log N}\, ,
$$
with $\nu^2$ as in~\eqref{eq:def-v}. For $\theta=\frac{\log N}{N}$, let $A_\theta$ be the set of $y\in\cH$ such that there exists a path from $x$ to $y$ which has probability larger than $\theta$ to be seen by a \nbrw\ of length $t$. Since, for all $y\in A_\theta$, we have $P^t(x,y)\geq \theta$, and since $P^t(x,\cdot)$ is a probability, the set $A_\theta$ has size less than $1/\theta$, hence
$$
\cD_x(t)\geq P^t(x,A_\theta)-\pi(A_\theta)\geq P^t(x,A_\theta)-\frac{1}{\theta N}\, \cdot 
$$
Taking expectation with respect to the pairing, we have
$$
\E P^t(x,A_\theta)\geq \P_x\left(\prod_{s=1}^t\frac{1}{\deg(X_s)}>\theta\right)=\P_x\left(\prod_{s=1}^t\frac{1}{\deg(X^\star_s)}>\theta\right)+o(1)\, ,
$$
where the last equality is by~\eqref{eq:coupling}. Using Lemma~\ref{lem:berry-esseen}, we have
$$
\P_x\left(\prod_{s=1}^t\frac{1}{\deg(X^\star_s)}>\theta\right)=\P_x\left(\frac{S_t-\mu t}{\nu\sqrt{t}} <-\lambda +o(1)\right)\geq \overline{\Phi}(\lambda)+O\left(\frac{1}{\sqrt{\alpha t}}\right)\, .
$$
Since $\alpha\gg \frac{1}{\log N}$ by assumption, we get
$$
\min_{x\in \cH}\E \cD_x(t)\geq \overline{\Phi}(\lambda) +o(1)\, .
$$
\subsection{Upper bound.}
\label{subsec:upper-bound-cutoff}

As in~\citep{ben-salez}, the first step is to reduce the maximization over all starting points to reasonably nice starting points, namely, to points whose neighborhood up to some level is a tree. We stress out that Lemma~\ref{lem:root} holds without any assumption on $\alpha$.

We call $x\in\cH$ a root if the ball of radius $R$ centered at $x$ (denoted by $\cB_x$) is a tree, where  
\begin{eqnarray}
\label{def:h}
R =  \left\lceil\frac{\log N}{6\log \Delta}\right\rceil.
\end{eqnarray}

We denote by $\cR$ the set of roots. The following lemma shows that we may restrict our attention to starting points in $\cR$. Its proof is very similar to the one of Proposition 4.1 in~\citep{ben-salez}, the introduction of communities only slightly changes the argument.

\begin{lemma}
\label{lem:root} 
Let $K=\lfloor \log\log N \rfloor$. Then
\begin{eqnarray*}
\max_{x\in\cH}P^K(x, \cH\setminus\cR) & \xrightarrow[]{\P} & 0.
\end{eqnarray*}
\end{lemma}

\begin{proof}[Proof of Lemma \ref{lem:root}]
Define $\ell=\left\lceil\frac{\log N}{5\log \Delta}\right\rceil$ and fix $x\in\cH$. The ball of radius $\ell$ around $x$ can be generated sequentially, its half-edges being given a type and then paired with a uniformly chosen other hitherto unpaired half-edge from the same or the other community depending on the type, until the entire ball is generated. Observe that at most $k=\frac{\Delta\left((\Delta-1)^{\ell}-1\right)}{\Delta-2}$ pairs are formed, and that, for each of them, the number of unpaired half-edges having an already paired neighbor is at most  $\Delta(\Delta-1)^\ell$. Hence, if the half-edge that is to be paired is in $\cH_i$, the conditional chance to pair it with a half-edge that has an already paired neighbor (thereby creating a cycle) is at most $\frac{\Delta(\Delta-1)^\ell-1}{N_i-2k-1}$ if it has been given an internal type, or $\frac{\Delta(\Delta-1)^\ell-1}{N_{1-i}-2k-1}$ if it has been given an outgoing type. Thus, letting $q$ be the minimum of those two ratios, the probability that more than one cycle is found is at most
\begin{eqnarray*}
(kq)^2 & = & O\left(\frac{\Delta^{4\ell}}{N^2}\right)\ = \ o\left(\frac{1}{N}\right),
\end{eqnarray*}
by the definition of $\ell$ and assumption~\eqref{eq:assumption-N0-N1}.
Summing over all $x\in\cH$ (union bound), we obtain that with probability $1-o(1)$, no ball of radius $\ell$ in $G$ contains more than one cycle. Now fix a graph $G$ with the above property. Then the \textsc{nbrw} on $G$ starting from any $x\in\cH$ will very quickly reach a root, namely it satisfies
\begin{eqnarray}
\label{toshow}
\P\left(X_K\not\in\cR\right) & \leq & 2^{1-K}=o(1),
\end{eqnarray}
by exactly the same argument as for the proof of equation $(21)$ in ~\citep{ben-salez}.
\end{proof}

We have
\begin{eqnarray*}
\cD(t+K)&\leq & \max_{x\in\cH}P^K(x, \cH\setminus\cR) + \max_{x\in\cR}\cD_x(t)\, .
\end{eqnarray*}
By Lemma~\ref{lem:root}, the first term is $o_{\P}(1)$, and, for all $x\in \cR$, bounding the summands corresponding to $y\in (\cH\setminus\cR)\cup \cB_x$ by $1/N$,
\begin{eqnarray*}
\cD_x(t) 
&\leq & \sum_{y\in\cR\setminus \cB_x}\left(\frac{1}{N}-P^t(x,\eta(y))\right)_+ +\frac{|(\cH\setminus\cR)\cup \cB_x|}{N}\, \cdot
\end{eqnarray*}
Observe that Lemma~\ref{lem:root} together with the fact that $P^K$ is doubly stochastic (since $P$ is) imply that 
\begin{eqnarray*}
|\cH\setminus \cR|=\sum_{x\in\cH}P^K(x,\cH\setminus \cR)=o_{\P}(N)\, .
\end{eqnarray*}
And for all $x\in \cR$, we have (deterministically) $|\cB_x|\leq \Delta^R\leq N^{1/6}$. Hence  
\begin{eqnarray*}
\max_{x\in\cR} \frac{|(\cH\setminus\cR)\cup \cB_x|}{N} =o_{\P}(1)\, .
\end{eqnarray*}

The following proposition will therefore conclude the proof of the upper bound.

\begin{proposition}
\label{prop:main}
For $t=\frac{\log N}{\mu}+(\lambda+o(1))\sqrt{\frac{\nu^2\log N}{\mu^3}}$, we have
\begin{eqnarray*}
\max_{x\in\cR}\sum_{y\in\cR\setminus\cB_x}\left(\frac{1}{N}-P^t(x,\eta(y)\right)_+ & \leq & \overline{\Phi}(\lambda)+o_{\P}(1)\, .
\end{eqnarray*}
\end{proposition}

To prove this, we first observe that, by property~\eqref{sym}, we can write
\begin{equation}\label{eq:crucial}
P^t(x,\eta(y))=\sum_{u,v} P^{t/2}(x,u)P^{t/2}(y,v)\ind_{\{\eta(u)=v\}}\, .
\end{equation}
We will show that, to approximate well this weighted sum of indicators, it is not necessary to reveal the entire balls of radius $t/2$ over $x$ and $y$. Instead, we consider an exploration process which generates the pairing $\eta$ along with two disjoint trees $\cT_x$ and $\cT_y$, rooted at $x$ and $y$ respectively. Initially, all half-edges are unpaired and no type has been revealed. Tree $\cT_x$ is reduced to $x$ and tree $\cT_y$ is reduced to $y$. Then at each time step,
\begin{enumerate}
\item An unpaired half-edge $z$ of $\cT_x\cup\cT_y$ is chosen, provided it satisfies
\begin{equation}\label{eq:cond-wmin-height}
\bw(z)\geq \bw_{\textsc{min}}=N^{-\frac{1}{2}-\frac{\log(2)}{16\log(\Delta)}}\quad\textrm{ and }\quad  \bh(z) < t/2\quad \, ,
\end{equation}
where $\bw(z)$ and $\bh(z)$ correspond to the \emph{weight} and \emph{height} of $z$,defined as follows: if $z\in\cT_r$ for $r\in\{x,y\}$, there is a unique path $(z_0,\dots,z_h)$ from $r$ to $z$, with $z_0=r$ and $z_h=z$. The value $h$ is then called the height of $z$, denoted $\bh(z)$, and its weight is
$$
\bw(z)=\prod_{i=1}^h\frac{1}{\deg(z_i)}\, \cdot 
$$
\item If $z\in\cH_i$ for $i\in\{0,1\}$, the type of $z$ is set to outgoing with probability proportional to $p$ minus the number of paired outgoing half-edges of $\cH_i$, and internal with probability proportional to $N_i-p$ minus the number of paired internal half-edges of $\cH_i$. 
\item If $z$ is internal, it is paired with $z'$, uniformly chosen among the unpaired half-edges of $\cH_i$, and the type of $z'$ is set to internal. If $z$ is outgoing, it is paired with $z'$, uniformly chosen among the unpaired half-edges of $\cH_{1-i}$, and the type of $z'$ is set to outgoing. 
\item If $z'$ was not already in $\cT_x\cup\cT_y$ and is not a neighbor of either $x$ or $y$, then the neighbors of $z'$ are added to $\cT_x\cup\cT_y$ as children of $z$. Otherwise, both $z$ and $z'$ are marked with the color \red.
\end{enumerate}
This exploration process continues until no unpaired half-edge in $\cT_x\cup\cT_y$ satisfies~\eqref{eq:cond-wmin-height}. The pairing $\eta$ is then completed to form the graph $G$. For $r\in\{x,y\}$, we denote by $\partial \cT_r$ the set of leaves of $\cT_r$, and by $\cF_r$ the subset of leaves of $\partial \cT_r$ which are at distance $t/2$ of $r$.

Note that, by~\eqref{eq:cond-wmin-height}, for $r\in\{x,y\}$,
$$
\frac{t}{2}=\sum_{k=1}^{t/2}\sum_{z\in\cT_r}\ind_{\{\bh(z)=k\}}\bw(z)\geq \left(|\cT_r|-1\right)\frac{\bw_{\textsc{min}}}{\Delta}\, ,
$$
which, together with~\eqref{eq:assumption-Delta}, implies
\begin{equation}\label{eq:bound-revealed-precise}
|\cT_x\cup\cT_y|=O\left(N^{\frac{1}{2}+\frac{\log(2)}{16\log(\Delta)}}\log N\right)=O\left(N^{\frac{1}{2}+\frac{\log(2)}{15\log(\Delta)}}\right)\, .
\end{equation}
In particular,
\begin{equation}\label{eq:bound-revealed}
|\cT_x\cup\cT_y|=O\left(N^{5/8}\right)\, .
\end{equation}

\begin{lemma}
\label{lem:walk-in-tree}
For all $\varepsilon>0$, with probability $1-o(1)$, for all $x\in\cR$ and $y\in\cR\setminus \cB_x$, we have
$$
\sum_{u\in\partial\cT_x\setminus \cF_x}\bw(u)+\sum_{v\in\partial\cT_y\setminus \cF_y}\bw(v)\leq \varepsilon\, .
$$
\end{lemma}
\begin{proof}[Proof of Lemma~\ref{lem:walk-in-tree}]
The trees' exploration can be stopped before height $t/2$ for two reasons: either the weight of the half-edge is too small, or it has been colored \red, namely, for $r\in\{x,y\}$,
$$
\sum_{u\in\partial\cT_r\setminus \cF_r}\bw(u)=\sum_{u\in\partial\cT_r}\bw(u)\ind_{\{\bw(u)<\bw_{\textsc{min}}\}} + \sum_{u\in\partial\cT_r}\bw(u)\ind_{\{\textrm{$u$ is \red}\}}\, .
$$
Let us first control the weight of \red\ half-edges. For $x\in\cR$ and $y\in\cR\setminus \cB_x$, all \red\ half-edges are at distance at least $R$ from $r$, and thus have weight smaller than $2^{-R}\leq N^{-\frac{\log(2)}{6\log(\Delta)}}$ by assumption~\eqref{eq:assumption-min-degree}. Moreover, by the same arguments as in the proof of Lemma~\ref{lem:root}, and using the upper bound~\eqref{eq:bound-revealed-precise}, the total number of \red\ half-edges in $\cT_r$ is stochastically dominated by twice a binomial random variable $\cB(k,q)$ where $k=O(N^{\frac{1}{2}+\frac{\log(2)}{15\log(\Delta)}})$ and $q=O(N^{-\frac{1}{2}+\frac{\log(2)}{15\log(\Delta)}})$. By Bennett's Inequality,
$$
\P\left(\sum_{u\in\partial\cT_r}\ind_{\{\textrm{$u$ is \red}\}}>N^{\frac{\log(2)}{7\log(\Delta)}}\right)\leq \exp\left(-\Omega\left(N^{\frac{\log(2)}{7\log(\Delta)}}\right)\right)\, .
$$
Hence, for all $\varepsilon>0$,
$$
\P\left(\exists x\in\cR, y\in\cR\setminus \cB_x, r\in\{x,y\},\, \sum_{u\in\partial\cT_r}\bw(u)\ind_{\{\textrm{$u$ is \red}\}}>\varepsilon\right)=o(1)\, .
$$
Let us now control the weight of paths with weight smaller than $\bw_{\textsc{min}}$. To this end, consider $m=\lfloor\log N\rfloor$ independent \textsc{nbrw}s on $G$ starting at $r$, each being stopped as soon as its weight falls below $\bw_{\textsc{min}}$, and let $A$ be the event that their trajectories form a tree of height less than $t/2$.  Clearly,
$$
\P\left(A \given G \right)\geq \left(\sum_{u\in\partial\cT_r}\bw(u)\ind_{\{\bw(u)<\bw_{\textsc{min}}\}}\right)^m\, .
$$
Taking expectation and using Markov inequality, we deduce that
$$
\P\left(\sum_{u\in\partial\cT_r}\bw(u)\ind_{\{\bw(u)<\bw_{\textsc{min}}\}}>\varepsilon\right)  \leq  \frac{\P\left(A\right)}{\varepsilon^m},
$$
where the average is now taken over both the  walks and the graph. To prove that the above probability is $o(1/N^2)$, it is enough to show that $\P(A)  =  o(1)^m.$ To do so, we generate the $m$ stopped \textsc{nbrw}s one after the other, revealing types and pairs along the way, as described in Section~\ref{sec:coupling}. Given that the first $\ell-1$ walks form a tree of height less than $t/2$, the conditional probability that the $\ell^\textrm{th}$ walk also fulfills the requirement is $o(1)$, uniformly in $1\leq\ell\leq m$. Indeed, 
\begin{itemize}
\item either it attains length $s=\lceil 4\log\log N\rceil$ before leaving the graph spanned by the first $\ell-1$ trajectories and reaching an unpaired half-edge: thanks to the tree structure, there are at most $\ell-1<m$ possible trajectories to follow, each having weight at most $2^{-s}$ by~\eqref{eq:assumption-min-degree}, so the conditional probability is at most $m2^{-s}=o(1).$
\item or the remainder of its trajectory after the first unpaired half-edge $z$ has weight less than $\Delta^s\bw_{\textsc{min}}$: this part consists of at most $t/2$ half-edges which can be coupled with $(X^\star_k)_{k=1}^{t/2}$ for a total-variation cost of $O({mt^2}/{N})$, and for $N$ large enough
$$
\P_z\left(\prod_{k=1}^{t/2}\frac{1}{\deg(X_k^\star)} \leq   \Delta^s\bw_{\textsc{min}}\right)\leq \P_z\left(S_{t/2}-\frac{\mu t}{2}\geq \frac{\log(2)}{18\log(\Delta)}\log N\right)\, ,
$$
which is $o(1)$ by Lemma~\ref{lem:berry-esseen}.
\end{itemize}
\end{proof}

For each $(i,j)\in\{0,1\}^2$, define
$$
\W_{i,j}=\sum_{u\in\cF_x\cap\cH_i}\sum_{v\in\cF_y\cap\cH_j}\bw(u)\bw(v)\, ,
$$
and, for $\theta=(N(\log N)^3)^{-1}$,
$$
\W^\theta_{i,j}=\sum_{u\in\cF_x\cap\cH_i}\sum_{v\in\cF_y\cap\cH_j}\bw(u)\bw(v)\ind_{\{\bw(u)\bw(v)\leq\theta\}}\, .
$$

\begin{lemma}
\label{lem:mix-communities}
For all $\varepsilon>0$, with probability $1-o(1)$, for all $x\in\cR$ and $y\in\cR\setminus \cB_x$,
$$
1\leq (N_0-p)\frac{N}{N_0^2}\W_{0,0} +(N_1-p)\frac{N}{N_1^2}\W_{1,1} +\frac{pN}{N_0 N_1}\left(\W_{0,1}+\W_{1,0}\right)+\varepsilon \, .
$$
\end{lemma}
\begin{proof}[Proof of Lemma~\ref{lem:mix-communities}]
Note that
$$
1= \frac{N_0-p}{N} +\frac{N_1-p}{N} +\frac{2p}{N}\, ,
$$
so that to prove the lemma, it is enough to establish that for all $\varepsilon$, with probability $1-o(1)$, for all $x\in\cR$ and $y\in\cR\setminus \cB_x$, for all $i,j\in\{0,1\}$, $\W_{i,j}\geq \frac{N_i N_j}{N^2}-\varepsilon$. Now, the event $\left\{\W_{i,j} < \frac{N_i N_j}{N^2}-\varepsilon\right\}$ is included in
$$
\left\{\sum_{u\in\cF_x\cap\cH_i}\bw(u)<\frac{N_i}{N}-\frac{\varepsilon}{2}\right\} \cup \left\{\sum_{v\in\cF_y\cap\cH_j}\bw(u)<\frac{N_j}{N}-\frac{\varepsilon}{2}\right\}\, .
$$
By Lemma~\ref{lem:walk-in-tree}, with probability $1-o(1)$, for all $x\in\cR$, $y\in\cR\setminus\cB_x$ and $r\in\{x,y\}$, we have $\sum_{u\in \cF_r}\bw(u)\geq 1-\varepsilon/4$, so that it remains to show that for all $\varepsilon>0$, $r\in\{x,y\}$, and $i\in\{0,1\}$,
$$
\P\left(\sum_{u\in\cF_r\cap\cH_i}\bw(u)>\frac{N_i}{N}+\varepsilon\right)=o\left(\frac{1}{N^2}\right)\, .
$$
To do so, we proceed as in the proof of Lemma~\ref{lem:walk-in-tree}. Consider $m=\lfloor (\log N)^2\rfloor$ independent \textsc{nbrw}s on $G$ starting at $r$, each of length $t/2$, and let $B$ be the event that their trajectories form a tree and that they all end in $\cH_i$. We have
$$
\P\left(\sum_{u\in\cF_r\cap\cH_i}\bw(u)>\frac{N_i}{N}+\varepsilon\right)  \leq  \frac{\P\left(B\right)}{\left(N_i/N+\varepsilon\right)^m},
$$
To prove that the above probability is $o(1/N^2)$, it is enough to show that
$\P(B)  \leq  \left(N_i/N+\varepsilon/2\right)^m.$
Generate the $m$ \textsc{nbrw}s one after the other, revealing types and pairs along the way, as described in Section~\ref{sec:coupling}. Given that the first $\ell-1$ walks form a tree and all end in $\cH_i$, the conditional probability that the $\ell^\textrm{th}$ walk also does is smaller than $N_i/N+\varepsilon/2$, uniformly in $1\leq\ell\leq m$. Indeed, 
\begin{itemize}
\item either it attains length $s=\lceil 4\log\log N\rceil$ before leaving the graph spanned by the first $\ell-1$ trajectories and reaching an unpaired half-edge: thanks to the tree structure, there are at most $\ell-1<m$ possible trajectories to follow, each having weight at most $2^{-s}$ by~\eqref{eq:assumption-min-degree}, so the conditional probability is at most $m2^{-s}=o(1).$
\item or it encounters an unpaired half-edge $z$ at some time $s'<s$ and the remainder of its trajectory can be coupled with $(X^\star_k)_{k=s'+1}^{t/2}$ for a total-variation cost of $O({mt^2}/{N})$. By~\eqref{eq:prob-coupling}, and since $t\gg 1/\alpha$,
$$
\P_z\left( X^\star_{t/2-s'}\in \cH_i \right)\leq N_i/N+\varepsilon/2 \, .
$$
\end{itemize}
\end{proof}

\begin{lemma}\label{lem:stein}
For all $\varepsilon>0$, with probability $1-o(1)$, for all $x\in\cR$ and $y\in\cR\setminus \cB_x$,
$$
NP^t(x,\eta(y))\geq (N_0-p)\frac{N}{N_0^2}\W^\theta_{0,0}+(N_1-p)\frac{N}{N_1^2}\W^\theta_{1,1}+\frac{pN}{N_0 N_1}\left(\W^\theta_{0,1}+\W^\theta_{1,0}\right)-\varepsilon\, .
$$
\end{lemma}
\begin{proof}[Proof of Lemma~\ref{lem:stein}]
In the sum \eqref{eq:crucial}, retaining only those paths that stay in $\cT_x\cup\cT_y$ and that have weight less than $\theta$, we have
\begin{equation*}
N P^t(x,\eta(y))\geq N \sum_{u\in\cF_x}\sum_{v\in\cF_y}\omega_{uv}\ind_{\{\eta(u)=v\}}\, ,
\end{equation*}
where $\omega_{uv}=\bw(u)\bw(v)\ind_{\{\bw(u)\bw(v)\leq\theta\}}$. Let us first condition on the types of the unpaired half-edges at the end of the exploration stage, and average over the remaining pairing. For $i\in\{0,1\}$, let $\cI_i$ (resp. $\cO_i$) be the set of unpaired internal (resp. outgoing) half-edges of $\cH_i$ at the end of the exploration stage. By applying  \citep[Lemma 6.1]{ben-salez} to the sets of unpaired internal half-edges (and noticing that, for us, the quantity denoted by $m$ there satisfies $\sum_{\substack{u\in\cF_x\cap\cI_i \\ v\in\cF_y\cap\cI_i}}\omega_{uv}\frac{N}{N_i-p}\leq m\leq \frac{N}{|\cI_i|-1}$), we have, for $i\in\{0,1\}$ and for all $\varepsilon>0$,
\begin{equation}\label{eq:eq1:proof-stein}
\P\left(N\!\!\!\!\sum_{\substack{u\in\cF_x\cap\cI_i \\ v\in\cF_y\cap\cI_i}}\!\!\omega_{uv}\left(\ind_{\{\eta(u)=v\}}-\frac{1}{N_i-p}\right)<-\varepsilon \, \Big| \, |\cI_i| \right)
\leq \exp\left(-\frac{\varepsilon^2(|\cI_i|-1)}{4\theta N^2}\right).
\end{equation}
Combining~\eqref{eq:bound-revealed}, \eqref{eq:assumption-N0-N1} and~\eqref{eq:assumption:p}, we have $|\cI_i|\asymp N_i -p\asymp N$, entailing that the right-hand side in~\eqref{eq:eq1:proof-stein} is $o(1/N^2)$. Now, applying  \citep[Proposition 1.1]{chatterjee2007stein} (or rather its refinement for the left tail given in its proof, with the summands here bounded by $N\theta$ instead of $1$ there, and noticing that, for us, the quantity denoted by $C=2\E X$ there satifies $\sum_{\substack{u\in\cF_x\cap\cO_i \\ v\in\cF_y\cap\cO_{1-i}}}\omega_{uv}\frac{N}{p}\leq \E X\leq \frac{N}{|\cO_i|}$), we have for $i\in\{0,1\}$ and for all $\varepsilon>0$,
\begin{equation}\label{eq:eq2:proof-stein}
\P\left(N \sum_{\substack{u\in\cF_x\cap\cO_i \\ v\in\cF_y\cap\cO_{1-i}}}\omega_{uv}\left(\ind_{\{\eta(u)=v\}}-\frac{1}{p}\right)<-\varepsilon \, \Big| \, |\cO_i| \right)
\leq \exp\left(-\frac{\varepsilon^2|\cO_i|}{4\theta N^2}\right)\, .
\end{equation}
Again, \eqref{eq:bound-revealed} yields $|\cO_i|\asymp p$, and since by assumption $p/N\gg 1/\log N$, the right-hand side in~\eqref{eq:eq2:proof-stein} is also $o(1/N^2)$. Our second task is to average over the types of half-edges in $\cF_x\cup\cF_y$. To this end, for $i\in\{0,1\}$, let $\cU_i$ be the set of unpaired half-edges of $\cH_i$ at the end of the exploration stage and write
$$
Y=\frac{N}{N_i-p}\sum_{\substack{u\in\cF_x\cap\cI_i \\ v\in\cF_y\cap\cI_i}}\omega_{uv}=\sum_{u,v\in (\cF_x \cup \cF_y)\cap\cU_i} q_{uv} B_u B_v\, ,
$$
where $q_{uv}=\frac{N}{N_i-p}\omega_{uv} \ind_{\{u\in\cF_x\}}\ind_{\{v\in\cF_y\}}$ and $B_u=\ind_{\{u\in\cI_i\}}$. Conditionally on $\cT_x$ and $\cT_y$, the sequence $(B_u)_{u\in (\cF_x \cup \cF_y)\cap\cU_i}$ enjoys a strong negative dependence property known as the strong Rayleigh property~\citep{borcea2009negative} (the sequence $(B_u)_{u\in \cU_i}$ enjoys it as a sequence of Bernoulli variables conditioned on its sum, and any subsequence of a strong Rayleigh sequence is also strong Rayleigh). Observing that $Y$ is a Lipschitz function of $(B_u)$ with constant 
$$
\frac{N}{N_i-p}\theta |\cF_x \cup \cF_y|=O( N^{-3/8})
$$ 
by~\eqref{eq:bound-revealed}. Applying \citep[Theorem 3.2]{pemantle2014concentration}, and using that, again thanks to~\eqref{eq:bound-revealed}, the quantity denoted by $\mu$ in their paper here is $O(N^{5/8})$, we have, for all $\varepsilon>0$,
$$
\bP\left(Y-\bE Y<-\varepsilon  \right)\leq \exp\left(-\Omega(N^{1/8})\right)\, ,
$$
where $\bP$ and $\bE$ are the probability law and expectation given $\cT_x\cup\cT_y$. Similarly, let
$$
Z=\frac{N}{p}\sum_{\substack{u\in\cF_x\cap\cO_i \\ v\in\cF_y\cap\cO_{1-i}}}\omega_{uv}=\sum_{u,v\in \cF_x \cup \cF_y} q'_{uv} B'_u B'_v\, ,
$$
where now $q'_{uv}=\frac{N}{p}\omega_{uv} \ind_{\{u\in\cF_x\cap\cH_i\}}\ind_{\{v\in\cF_y\cap\cH_{1-i}\}}$ and $B'_u=\ind_{\{u\in\cO_i\cup\cO_{1-i}\}}$. The sequence $(B'_u)_{u\in \cF_x \cup \cF_y}$ still enjoys the strong Rayleigh property (the sequences $(B'_u)_{u\in\cU_i}$ and $(B'_u)_{u\in\cU_{1-i}}$ both enjoy it as sequences of Bernoulli conditioned on their sum and the concatenation of two independent strong Rayleigh sequences is also strong Rayleigh; and, as already mentioned, if a sequence is strong Rayleigh, any of its subsequences is too). The variable $Z$ is a Lipschitz function with constant $\frac{N}{p}\theta |\cF_x \cup \cF_y|=O(N^{-3/8})$ by~\eqref{eq:bound-revealed} and our assumption that $p\gg \frac{N}{\log N}$. Hence another application of \citep[Theorem 3.2]{pemantle2014concentration} yields
$$
\bP\left(Z-\bE Z<-\varepsilon  \right)\leq \exp\left(-\Omega(N^{1/8})\right)\, .
$$
The proof is then concluded by noticing that
$$
\bE Y= (1+o(1))(N_i-p)\frac{N}{N_i^2}W^\theta_{i,i}\, ,
$$
and
$$
\bE Z= (1+o(1))\frac{pN}{N_0 N_1} W^\theta_{i,1-i}\, ,
$$
\end{proof}

Combining Lemma~\ref{lem:mix-communities} and~\ref{lem:stein}, we obtain that for all $\varepsilon>0$, with probability $1-o(1)$, for all $x\in\cR$, and $y\in\cR\setminus\cB_x$,
$$
\left(1-N P^t(x,\eta(y)\right)_+ \leq (N_0-p)\frac{N}{N_0^2}\overline{\W}^\theta_{0,0}+(N_1-p)\frac{N}{N_1^2}\overline{\W}^\theta_{1,1}+\frac{pN}{N_0 N_1}\left(\overline{\W}^\theta_{0,1}+\overline{\W}^\theta_{1,0}\right)+\varepsilon\, ,
$$
where $\overline{\W}^\theta_{i,j}=\W_{i,j}-\W^\theta_{i,j}$. The proof of Proposition~\ref{prop:main} will then be concluded by the following lemma.

\begin{lemma}
\label{lem:final}
For all $\varepsilon>0$, with probability $1-o(1)$, for all $x\in\cR$, and $y\in\cR\setminus\cB_x$, for all $i,j\in\{0,1\}$,
$$
\overline{\W}^\theta_{i,j}\leq \frac{N_i N_j}{N^2}\overline{\Phi}(\lambda)+\varepsilon\, .
$$
\end{lemma}
\begin{proof}[Proof of Lemma~\ref{lem:final}]
Set $m=\lceil (\log N)^2\rceil$ and let $X^{(1)},\dots,X^{(m)}$ be $m$ independent \textsc{nbrw}s of length $t/2$ started at $x$, and $Y^{(1)},\dots,Y^{(m)}$ be $m$ independent \textsc{nbrw}s of length $t/2$ started at $y$, independent of $X^{(1)},\dots,X^{(m)}$. Let $C$ denote the event that their trajectories form a cycle-free graph and that for all $1\leq k\leq m$, $X^{(k)}_{t/2}\in\cH_i$, $Y^{(k)}_{t/2}\in\cH_j$, and
$$
\prod_{\ell=1}^{t/2-\Lambda \alpha^{-1}} \frac{1}{\deg(X^{(k)}_\ell)}\prod_{\ell=1}^{t/2-\Lambda \alpha^{-1}}\frac{1}{\deg(Y^{(k)}_\ell)}>\theta\, ,
$$ 
for some constant $\Lambda>0$ to be specified later (note that by our assumption on $\alpha$, the term $\alpha^{-1}$ grows much more slowly than the window of order $\sqrt{\frac{\log N}{\alpha}}$). Then, $P\left(C \given G \right)\geq \left(\overline{\W}^\theta_{i,j}\right)^m$, and
\begin{eqnarray*}
\P\left(\overline{\W}^\theta_{i,j}>\frac{N_i N_j}{N^2}\overline{\Phi}(\lambda)+\varepsilon\right) & \leq & \frac{\P\left(C\right)}{\left(\frac{N_i N_j}{N^2}\overline{\Phi}(\lambda)+\varepsilon\right)^m}.
\end{eqnarray*}
Generate the $2m$ walks $X^{(1)},Y^{(1)},\ldots,X^{(m)},Y^{(m)}$ one after the other along with the underlying types and pairs, as above. Given that the first $\ell -1$ pairs already satisfy the desired property, the conditional chance that $X^{(\ell)},Y^{(\ell)}$ also does is at most $\frac{N_i N_j}{N^2}\overline{\Phi}(\lambda)+\varepsilon/2$, uniformly in $1\leq \ell \leq m$. Indeed,
\begin{itemize}
\item either one of the two walks attains length $s=\lceil 4\log\log N\rceil$ before leaving the graph spanned by the first $2(\ell-1)$ trajectories and reaching an unpaired half-edge: the conditional chance is at most $2m2^{-s}=o(1)$.
\item or they both leave the graph before $s$: the remainder of their trajectory can then be coupled with $(X^\star_k)$ and $(Y^\star_k)$ for a total-variation cost of $O(mt^2/N)$. Thus, it is enough to bound, uniformly in $x,y,z,z'\in\cH$,
$$
\P_{x,y}\left(\prod_{k=1}^{t_\star/2}\deg(X^\star_k)\deg(Y^\star_k)<\frac{1}{\theta}\right)\P_z\left(X^\star_{\Lambda\alpha^{-1}}\in\cH_i\right)\P_{z'}\left(X^\star_{\Lambda\alpha^{-1}}\in\cH_j\right)\, ,
$$
where $t_\star/2=t/2-s-\Lambda \alpha^{-1}$. By~\eqref{eq:prob-coupling}, the constant $\Lambda$ can be chosen large enough so that for all $z\in\cH$, $\P_z\left(X^\star_{\Lambda\alpha^{-1}}\in\cH_i\right)\leq \frac{N_i}{N}+\varepsilon/8$. Also, it is note hard to check that the mixing time of $(X^\star_k)$ is of order $1/\alpha$, and, since $1/\alpha\ll t_\star$, the total-variation distance between the law of $X^\star_{t_\star/2+1}$ and the law $Y^\star_{t_\star/2}$ is $o(1)$. And by reversibility of the chain $(X^\star_k)$, we can write
$$
\P_{x,y}\left(\prod_{k=1}^{t_\star/2}\deg(X^\star_k)\deg(Y^\star_k)<\frac{1}{\theta}\right)\leq \P_{x}\left(\prod_{k=1}^{t_\star}\deg(X^\star_k)<\frac{1}{\theta}\right)+o(1)\, .
$$
Finally, by Lemma~\ref{lem:berry-esseen}, 
$$
\P_{x}\left(\prod_{k=1}^{t_\star}\deg(X^\star_k)<\frac{1}{\theta}\right)=\overline{\Phi}(\lambda)+o(1)\, .
$$
\end{itemize}
\end{proof}

\section{Proof of Theorem~\ref{thm:no-cutoff}}
\label{sec:no-cutoff}

In this whole section, we now assume that $\alpha =O(\frac{1}{\log N})$. However, both for the lower and upper bounds, we will use different kinds of arguments according to the more precise decay of $\alpha$ in this regime. 

Let us define the two probability measures $\pi_0$ and $\pi_1$ on $\cH$ by
$$
\pi_0(x)=
\begin{cases}
\frac{1}{N_0} & \mbox{if $x\in\cH_0$,}\\
0 & \mbox{otherwise,}
\end{cases}
\quad 
\text{and }
\quad 
\pi_1(x)=
\begin{cases}
\frac{1}{N_1} & \mbox{if $x\in\cH_1$,}\\
0 & \mbox{otherwise.}
\end{cases}
$$
For $i=0,1$, let us denote by $\cO_i$ the set of outgoing half-edges of $\cH_i$, and let $\tau_i=\min\{t\geq 0,\, X_t\in \cO_i\}$.
Without loss of generality, we consider starting points in $\cH_0$. 

Let us first state a lemma that will be used both for the upper and lower bound. 

\begin{lemma}\label{lem:mix-one-community}
Assume $\alpha\ll \frac{1}{\log N}$. Then, for $s=2\log N$, for all $x\in\cH_0$,
\[
\| \P_x^G(X_s\in\cdot)-\pi_0\|_\tv = o_\P(1)\, .
\]
\end{lemma}
\begin{proof}[Proof of Lemma~\ref{lem:mix-one-community}]
Let us define the random graph $G_0$ formed with the half-edges of $\cH_0$ as follows: the internal half-edges of $\cH_0$ are paired exactly as in $G$, and the outgoing half-edges of $\cH_0$ are paired uniformly at random within each other (recall that $p$ is even). Since outgoing half-edges are chosen uniformly at random, the graph $G_0$ is exactly distributed according to the configuration model on $\cH_0$ with uniform pairing. By the triangle inequality,
\[
\| \P_x^G(X_{s}\in\cdot)-\pi_0\|_\tv \leq \| \P_x^G(X_{s}\in\cdot)-\P_x^{G_0}(X_{s}\in\cdot)\|_\tv +\| \P_x^{G_0}(X_{s}\in\cdot)-\pi_0\|_\tv 
\]
By Theorem 1.1 of~\citep{ben-salez}, with high probability, the \nbrw\ on $G_0$ has (worst case) cutoff at time $\frac{\log N_0}{\mu_0}$. Since all degrees are at least $3$, $\mu_0^{-1}\leq (\log 2)^{-1}< 2$. Hence for $s=2\log N$, 
\[
\max_{x\in\cH_0}\| \P_x^{G_0}(X_{s}\in\cdot)-\pi_0\|_\tv =o_\P(1)\, .
\]
On the other hand,
\begin{equation}\label{eq:no-passage}
\| \P_x^G(X_{s}\in\cdot)-\P_x^{G_0}(X_{s}\in\cdot)\|_\tv\leq \P^G_x(\tau_0 <s)=o_\P(1)\, ,
\end{equation}
since $\E\left[ \P^G_x(\tau_0 <s) \right]=O(s\alpha)=o(1)$, by our assumption on $\alpha$. 
\end{proof}

\begin{remark}\label{rem}
Note that the proof of Lemma~\ref{lem:mix-one-community} established that for all $x\in\cH_0$, $\| \P_x^G(X_s\in\cdot)-\pi_0\|_\tv \leq D_x$ where $D_x$ is a random variable that is determined only by the matching on the internal half-edges of $\cH_0$, and that $\max_{x\in\cH_0}\E[D_x]=o(1)$. 
\end{remark}

\subsection{Lower bound}
\label{subsec:lower-no-cutoff}

The proof of the lower bound is divided into two parts: first, we consider $\alpha\gg \frac{1}{\sqrt{N}}$. In this case, the coupling of Section~\ref{sec:coupling} still holds up to the mixing time and a tighter dependence in $\varepsilon$ can be obtained for $\tmix^{(x)}(\varepsilon)$. Then, using a simple conductance argument, we establish a lower bound of order $1/\alpha$ (but with looser dependence in $\varepsilon$), which holds as soon as $\alpha\ll \frac{1}{\log N}$.\\

\subsubsection{{\bf The case $\alpha\gg \frac{1}{\sqrt{N}}$}}

In this section, we show that if $\alpha\gg \frac{1}{\sqrt{N}}$, then for all $x\in\cH_0$,  
\[
\tmix^{(x)}(\varepsilon) \geq \frac{1+o_\P(1)}{\alpha}\log\left(\frac{N_1}{N\varepsilon}\right)\, . 
\]
Let $\delta>0$ and $t=\frac{1}{\alpha}  \log\left(\frac{N_1}{N(\varepsilon+\delta)}\right)$. We have
\[
\cD_x(t)\geq \P_x^G(X_t\in\cH_0)-\pi(\cH_0)=\P_x^G(X_t\in\cH_0)-\frac{N_0}{N}\, \cdot 
\]
Since $t\ll \sqrt{N}$, the coupling of Section~\ref{sec:coupling} holds up to $t$ with high probability, and equation~\eqref{eq:prob-coupling} gives
\[
\E\left[\P_x^G(X_t\in\cH_0)\right]=\P_x(X_t^\star\in\cH_0)+o(1)=\frac{N_0}{N}+\frac{N_1}{N}(1-\alpha)^t +o(1)\, .
\]
For large enough $N$, 
$$\E\left[\P_x^G(X_t\in\cH_0)\right]\geq \frac{N_0}{N}+\varepsilon +\frac{\delta}{2}\, \cdot $$
Hence, by Tchebytchev Inequality,
\[
\P\left(\cD_x(t)\leq \varepsilon\right)\leq \frac{4 \var\left( \P_x^G(X_t\in\cH_0) \right)}{\delta^2}\, \cdot 
\]
To conclude the lower bound, let us show that $\var\left( \P_x^G(X_t\in\cH_0)\right)=o(1)$. The second moment $\E\left[ \P_x^G(X_t\in\cH_0)^2\right]$ corresponds to the annealed probability that two independent \nbrw\ started at $x$ are in $\cH_0$ at time $t$. Generating the two walks one after the other, and noticing that the probability that the first exits $\cH_0$ before time $s=\lfloor \log\log N\rfloor$, or that the second follows the first for more than $s$ steps is $o(1)$, we obtain
\begin{align*}
\E\left[ \P_x^G(X_t\in\cH_0)^2\right] & =\left(\frac{N_0}{N}+\frac{N_1}{N}(1-\alpha)^t +o(1)\right)\left(\frac{N_0}{N}+\frac{N_1}{N}(1-\alpha)^{t-s} +o(1)\right)\\
&=\E\left[ \P_x^G(X_t\in\cH_0)\right]^2 +o(1)\, .
\end{align*}

\subsubsection{{\bf The case $\alpha\ll \frac{1}{\log(N)}$}}

Let us recall that the conductance $\bphi(S)$ of a set $S\subset \cH$ is defined as
 $$
 \bphi(S)=\frac{\sum_{x\in S}\sum_{y\in S^c}\pi(x)P(x,y)}{\sum_{x\in S}\pi(x)}\, \cdot 
 $$
 Observe that
 $$
 \bphi(\cH_0)=\alpha_0\, ,\quad\mbox{and}\quad \bphi(\cH_1)=\alpha_1\, \cdot 
 $$
By the triangle inequality, we have, for all $t\geq 0$,
$$
\cD_x(t)\geq \| \P_{\pi_0}^G(X_t\in \cdot) -\pi\|_\tv - \| \P_{x}^G(X_t\in \cdot) - \P_{\pi_0}^G(X_t\in \cdot) \|_\tv\, .
$$
On the one hand, for $s=2\log N$, and $t\geq s$,
\[
\| \P_{x}^G(X_t\in \cdot) - \P_{\pi_0}^G(X_t\in \cdot) \|_\tv \leq \| \P_{x}^G(X_s\in \cdot) - \P_{\pi_0}^G(X_s\in \cdot) \|_\tv=o_\P(1)\, .
\]
Indeed, by Lemma~\ref{lem:mix-one-community}, $ \| \P_{x}^G(X_s\in \cdot) - \pi_0 \|_\tv=o_\P(1)$, and, by equation~\ref{eq:no-passage},
\[
 \| \P_{\pi_0}^G(X_s\in \cdot) -\pi_0 \|_\tv\leq \P_{\pi_0}^G(\tau_0<s)=o(1)\, .
 \]
On the other hand,  by~\citep[equation (7.10)]{levin2017markov},
$$
\| \P_{\pi_0}^G(X_t\in \cdot) -\pi\|_\tv\geq \frac{N_1}{N}-\P_{\pi_0}^G(X_t\in \cH_1)\geq \frac{N_1}{N}- t\bphi(\cH_0)\, .
$$
Hence, for all $x\in\cH_0$, 
\[
\tmix^{(x)}(\varepsilon) \geq \frac{1}{\alpha_0}\left(\frac{N_1}{N}-\varepsilon-o_\P(1)\right)\, .
\]

\subsection{Upper bound}
\label{subsec:upper-no-cutoff}

As for the lower bound, the proof of the upper bound uses different arguments according to the decay of $\alpha$. We first consider $\alpha\ll \frac{1}{\log N}$, and then $\alpha\asymp \frac{1}{\log N}$. \\

\subsubsection{{\bf The case $\alpha\ll \frac{1}{\log(N)}$}}

Let $t=\frac{A}{\alpha\varepsilon}$ with $A>0$ some large constant to be specified later, and $s=\lceil 2\log N\rceil $. By the triangle inequality, for all $x\in\cH_0$,

\begin{eqnarray*}
\cD_{x}(t+2s)&\leq & \|\P_x^G(X_s\in\cdot)-\pi_0\|_\tv + \|\P_{\pi_0}^G(X_{t+s}\in\cdot)-\pi\|_\tv\, .
\end{eqnarray*}

By Lemma~\ref{lem:mix-one-community}, $\|\P_x^G(X_s\in\cdot)-\pi_0\|_\tv=o_\P(1)$.
Now observe that
\[
\|\P_{\pi_0}^G(X_{t+s}\in\cdot)-\pi\|_\tv =\frac{N_1}{N}\|\P_{\pi_0}^G(X_{t+s}\in\cdot)-\P_{\pi_1}^G(X_{t+s}\in\cdot)\|_\tv\, .
\]
The rest of the proof now follows from a coupling argument. Let $(X_k)$ and $(Y_k)$ be two random walks started at $\pi_0$ and $\pi_1$ respectively. First the two random walks evolve independently until the first time $\tau$ when they are in the same community. From time $\tau$, if they are both in $\cH_i$, if $X_\tau=x$ and $Y_\tau=y$, then the walks are coupled according to the otimal coupling for the distance at time $s$, that is, the coupling which attains $\P^{G}_{x,y}(X_s\neq Y_s)=\| \P_{x}^{G}(X_s\in\cdot)- \P_{y}^{G}(Y_s\in\cdot)\|_\tv$. If one of them switch community before they have met, then we repeat the same coupling, until they meet. Once they meet, they evolve together. By the strong Markov property, we have 
\begin{align*}
&\|\P_{\pi_0}^G(X_{t+s}\in\cdot)-\P_{\pi_1}^G(Y_{t+s}\in\cdot)\|_\tv\\
&\quad \leq \P^G_{\pi_0,\pi_1}(\tau>t) +\sum_{k=1}^t\sum_{x,y\in\cH} \P^G_{\pi_0,\pi_1}(\tau=k, X_\tau=x, Y_\tau=y)\P^G_{x,y}(X_{t+s-k}\neq Y_{t+s-k})\\
&\quad \leq \P^G_{\pi_0,\pi_1}(\tau>t) + \sum_{x,y\in\cH} \P^G_{\pi_0,\pi_1}(X_\tau=x, Y_\tau=y)\P^G_{x,y}(X_{s}\neq Y_{s})\, .
\end{align*}
Now,
\[
\P^G_{\pi_0,\pi_1}(\tau>t)\leq \P_{\pi_0}^G(\tau_0\geq t)+ \P_{\pi_1}^G(\tau_1\geq t) +\P^G_{\pi_0,\pi_1}(\tau_0=\tau_1)\, .
\]
Clearly, $\P^G_{\pi_0,\pi_1}(\tau_0=\tau_1)=o_\P(1)$ since the expectation is $O(\alpha)$. The following lemma states that, by times of order $1/\alpha$, both walks have hit an outgoing half-edge.

\begin{lemma}\label{lem:passage}
For $t=\frac{A}{\alpha\varepsilon}$ with $A>0$ large enough, 
\[
\P_{\pi_0}^G(\tau_0\geq t)\leq \varepsilon+o_\P(1) \quad \mbox{and} \quad \P_{\pi_1}^G(\tau_1\geq t)\leq \varepsilon+o_\P(1)\, .
\]
\end{lemma}
\begin{proof}[Proof of Lemma~\ref{lem:passage}]
First note that
$$\P_{\pi_0}^G(\tau_0\geq t)=\P_{x_0}^{G_0}(\tau_0\geq t)\, .$$ 
Now let $f$ be the function defined on $\cH_0$ by $f(z)=\alpha_0-\ind_{z\in\cO_0}$. Note that 
\[
\E_{\pi_0}^{G_0} f=0\, \quad\text{ and } \E_{\pi_0}^{G_0}f^2=\alpha_0(1-\alpha_0)\leq\alpha_0\, .
\]
We have
\[
\P_{\pi_0}^{G_0}(\tau_0\geq t)=\P_{\pi_0}^{G_0}\left(\frac{1}{t}\sum_{i=0}^{t-1} \ind_{X_i\in\cO_0}=0\right)\leq \P_{\pi_0}^{G_0}\left(\frac{1}{t}\sum_{i=0}^{t-1} f(X_i)\geq \alpha_0\right)\, ,
\]
and by Markov Inequality and Cauchy-Schwarz Inequality,
\begin{align*}
\P_{\pi_0}^{G_0}\left(\frac{1}{t}\sum_{i=0}^{t-1} f(X_i)\geq \alpha_0\right)& \leq \frac{1}{(t\alpha_0)^2} \left(t \E_{\pi_0}^{G_0}f^2+\sum_{k=1}^{t-1}(t-k) \E_{\pi_0}^{G_0}\left[ f(X_0)f(X_k)\right]\right)\\
&\leq \frac{1}{t\alpha_0} +\frac{1}{t\alpha_0^2}\sum_{k=1}^{t-1} \sqrt{\E_{\pi_0}^{G_0}[f^2] \E_{\pi_0}^{G_0}[f(X_k)^2]}\, .
\end{align*}
Now, by contraction, $\E_{\pi_0}^{G_0}[f(X_k)^2]\leq (1-\gamma_0)^t \E_{\pi_0}^{G_0}[f^2]$, where $\gamma_0$ is the Poincar\'e constant of the \nbrw\ on $G_0$. Using Cheeger Inequality and a conductance argument (as in \citep[Lemma 3.5]{benjamini2014mixing} for the conductance of the simple random walk), we have that with high probability, $\gamma_0$ is bounded away from $0$. Hence 
\[
\P_{\pi_0}^{G_0}(\tau_0\geq t)\leq \frac{1}{t\alpha_0} +\frac{1}{t\alpha_0}\sum_{k=1}^{t-1}(1-\gamma_0)^{t/2}\leq \frac{1}{t\alpha_0}\left(1+\frac{1}{1-\sqrt{1-\gamma_0}}\right)\, .
\]
For $t=\frac{A}{\alpha\varepsilon}$ and $A$ large enough, $\P_{\pi_0}^{G_0}(\tau_0\geq t)\leq \varepsilon$.
\end{proof}

To conclude the proof, we need to show that, with high probability, once the walks are in the same community, they couple in logarithmic time. As already observed, $\P^G_{\pi_0,\pi_1}(\tau_0=\tau_1)=o_\P(1)$, and on the event $\tau_0\neq \tau_1$, we can write
\begin{align*}
& \sum_{x,y\in\cH} \P^G_{\pi_0,\pi_1}(X_\tau=x, Y_\tau=y, \tau_0\neq\tau_1)\P^G_{x,y}(X_{s}\neq Y_{s})\\
& \quad =\sum_{x,y\in\cH_0}\frac{1}{N_0}\P^G_{\pi_1}(Y_{\tau_1+1}=y)\P_{x,y}^G(X_s\neq Y_s) + \sum_{x,y\in\cH_1}\frac{1}{N_1}\P^G_{\pi_0}(X_{\tau_0+1}=x)\P_{x,y}^G(X_s\neq Y_s)\, .
\end{align*}
For $x,y\in\cH_0$, using the definition of the coupling, the triangle inequality, and Remark~\ref{rem},
\begin{align*}
\P_{x,y}^G(X_s\neq Y_s)&=\| \P_{x}^{G}(X_s\in\cdot)- \P_{y}^{G}(Y_s\in\cdot)\|_\tv\leq D_x +D_y\, .
\end{align*}
Using the fact that, for $z\in\cH_1$, the variable $\P^G_{\pi_1}(Y_{\tau_1}=z)$ is independent from the matching on the internal half-edges of $\cH_0$, we obtain
\begin{align*}
\E\left[\sum_{x,y\in\cH_0}\frac{\P^G_{\pi_1}(Y_{\tau_1+1}=y)}{N_0}\P_{x,y}^G(X_s\neq Y_s)\right]&\leq \E\left[ \sum_{\substack{x\in\cH_0\\ z\in\cH_1}} \frac{\P^G_{\pi_1}(Y_{\tau_1}=z)}{N_0} \sum_{y\sim \eta(z)}\frac{1}{\deg(\eta(z))}(D_x+D_y)\right]\\
& \leq \sum_{z\in\cH_1}\E\left[\P^G_{\pi_1}(Y_{\tau_1}=z)\right]\Delta\sup_{x\in\cH_0}\E[D_x]\, ,
\end{align*}
which is $o(1)$ by Lemma~\ref{lem:mix-one-community}. Similarly,
\[
\E\left[  \sum_{x,y\in\cH_1}\frac{1}{N_1}\P^G_{\pi_0}(X_{\tau_0+1}=x)\P_{x,y}^G(X_s\neq Y_s) \right]=o(1)\, ,
\]
and this concludes the proof of the upper bound when $\alpha\ll \frac{1}{\log N}$.\\

\subsubsection{{\bf The case $\alpha\asymp \frac{1}{\log(N)}$}}

Take now $\varepsilon<N_1/N$ and $t=\frac{A}{\alpha}\log(1/\varepsilon)$, for $A>0$ to be specified later. Bounding the $\ell^1(\pi)$-distance by the $\ell^2(\pi)$-distance, we have
\[
\cD_x(t)\leq \frac{1}{2}\sqrt{ N\sum_{y\in\cH}P^t(x,y)^2 -1}\, .
\]
Hence
\[
\P\left(\cD_x(t)>\varepsilon\right)\leq \P\left(\sum_{y\in\cH} P^t(x,y)^2>\frac{1+4\varepsilon^2}{N}\right)\, .
\]
Note that
\[
\E\left[\sum_{y\in\cH}P^t(x,y)^2\right]=\P_x(X_t=Y_t)\, ,
\]
where $X$ and $Y$ are two independent \nbrw\ of length $t$ started at $x$. Recall that the \emph{excess} of a graph is the maximal number of edges that can be removed from it while keeping it connected. To estimate the annealed probability that $X_t=Y_t$, we distinguish different cases:
\begin{enumerate}
\item either the graph spanned by $X$ and $Y$ has excess strictly larger than $1$: this has probability $O\left(\frac{t^4}{N^2}\right)$, which is $o(1/N)$.
\item or $Y$ follows the trajectory of $X$ up to $t$: the probability is $O(2^{-t})$ which is $o(1/N)$ for $A$ large enough. 
\item or each trajectory cycle-free and there exist $s$ and $r$ with $s+r<t$ such that $Y$ follows $X$ for the first $s$ steps, then parts and merges back with $X$ for the last $r$ steps. In this situation, if at time $s+1$ both walks are in $\cH_i$, the probability that they merge back at time $t-r+1$ can be approximated by
\end{enumerate}
\[
\left(\P_{\pi_i}(X^\star_{t-s-r}\in \cH_0)^2 +o(1)\right)\frac{1}{N_0} + \left(\P_{\pi_i}(X^\star_{t-s-r}\in \cH_1)^2 +o(1)\right)\frac{1}{N_1}
\]
Summing over $s$ and $r$ such that $s+r<t$, observing that the contribution from $s$ or $r$ larger than $\log \log N$ is negligible, and that, with high probability, for $s\leq \log\log N$, both walks are still in $\cH_0$ at time $s+1$ (recall that we assumed $x\in\cH_0$), we obtain that 
\[
\P_x(X_t=Y_t)=\left(\frac{N_0}{N}+\frac{N_1}{N}(1-\alpha)^t\right)^2\frac{1+o(1)}{N_0}
+\left(\frac{N_1}{N}-\frac{N_1}{N}(1-\alpha)^t\right)^2\frac{1+o(1)}{N_1}\, .
\]
Using that $(1-\alpha)^t= \varepsilon^A+o(1)$, we have
\[
\P_x(X_t=Y_t)= \frac{1+\frac{N_1}{N}\left(\frac{N_1}{N_0}+1\right)\varepsilon^{2A} +o(1)}{N}\leq \frac{1+2\varepsilon^2}{N}\, ,
\]
for $A$ and $N$ large enough. By Tchebychev Inequality,
\[
\P\left(\cD_x(t)>\varepsilon\right)\leq \frac{N^2}{1+2\varepsilon^2} \var\left(\sum_{y\in\cH}P^t(x,y)^2\right)\, \cdot 
\]
To conclude the proof, let us show that $\var\left(\sum_{y\in\cH}P^t(x,y)^2\right)=o(N^{-2})$. Proceeding as above, we first note that 
\[
\E\left[ \left(\sum_{y\in\cH}P^t(x,y)^2\right)^2\right]=\P_x(X_t=Y_t, X'_t=Y'_t)\, ,
\]
where $X, Y, X'$ and $Y'$ are independent \nbrw\ of length $t$ started at $x$. Let us again distinguish different cases:
\begin{enumerate}
\item either the graph spanned by the four walks has excess strictly larger than $2$: this has probability $O\left(\frac{t^6}{N^3}\right)$, which is $o(1/N^2)$.
\item or there exist one walk which evolves only on the graph spanned by the three other walks: since the excess is at most $2$, there are not many possible trajectories to follow and the probability is $O(2^{-t})$, which is $o(1/N^2)$ for $A$ large enough.
\item or in the last situation, there exist $s,r$ and $s',r'$ with $s+r<t$ and $s'+r'<t$ such that $Y$ (resp. $Y'$) follows $X$ (resp. $X'$) for the first $s$ (resp. $s'$) steps, then parts and merges back with $X$ (resp. $X'$) for the last $r$ (resp. $r'$) steps.
\end{enumerate}
Denoting by $s^\star$ the minimum between $t$ and the first time when all four walks are distinct, we observe that, when summing over $s,r,s',r'$, the contribution from $s^\star$ or $r$ $r'$ larger than $\log \log N$ is negligible, and that, with high probability, for $s^\star\leq  \log\log N$, all walks are still in $\cH_0$ at time $s^\star$, we obtain that 
\begin{align*}
\P_x(X_t=Y_t, X'_t=Y'_t)&=\left(\left(\frac{N_0}{N}+\frac{N_1}{N}(1-\alpha)^t\right)^2\frac{1+o(1)}{N_0}
+\left(\frac{N_1}{N}-\frac{N_1}{N}(1-\alpha)^t\right)^2\frac{1+o(1)}{N_1}\right)^2\\
&= \P_x(X_t=Y_t)^2 +o\left(\frac{1}{N^2}\right)\, ,
\end{align*}
and this concludes the proof.\\
\medskip

\paragraph*{{\bf Acknowledgments.}}  The author would like to thank two anonymous referees for their valuable comments and helpful suggestions.

\bibliographystyle{abbrvnat}
\bibliography{biblio}

\end{document}